\documentclass[12pt]{amsart} 
\usepackage{amssymb}
\usepackage{amsmath}
\usepackage{latexsym}
\usepackage{amscd}
\usepackage{amsfonts}
\usepackage{pb-diagram}
\usepackage[all]{xypic}
\usepackage[mathcal,mathscr]{eucal}
\usepackage{amsthm}
\usepackage{graphicx}
\usepackage{accents}

\usepackage{url}
\usepackage{pdfsync}
\usepackage{bbm}
\usepackage{xypic}
\usepackage[pdftex,colorlinks]{hyperref}
\usepackage{pifont}
\usepackage{amsbsy}
\usepackage{epstopdf}
\usepackage{times}
\usepackage{epic,eepic}

\newtheorem{theorem}{Theorem}[section]

\newtheorem{lemma}[theorem]{Lemma}

\theoremstyle{remark}
\newtheorem{remark}[theorem]{\bf {Remark}}

\numberwithin{equation}{section}

\DeclareMathOperator{\Pf}{Pf}

\DeclareMathOperator{\TPf}{TPf}

\newcommand\bR{{\mathbb R}}

\newcommand{\ra}{\rightarrow}

\begin{document}
\title{ A short proof of a Theorem by Hopf}
\begin{abstract} A   proof based on the Chern-Gauss-Bonnet Theorem is given to Hopf Theorem concerning the degree of the Gauss map of a hypersurface in $\bR^n$. 
\end{abstract}

\author{Daniel Cibotaru}
\address{Universidade Federal do Cear\'a, Fortaleza, CE, Brasil}
\email{daniel@mat.ufc.br} 

\maketitle

\section{Introduction}

Let $H\subset \bR^n$ be a smooth,  compact,  embedded hypersurface without boundary. It is well-known that $H$ separates Euclidean space into two connected components. Let $W$ be the bounded component of $\bR^n$ with $\partial W=H$. Orient $H$ as the boundary of $W$ using the outer normal first convention. Hopf \cite{H} proved the following:
\begin{theorem}[Hopf] \label{HT} The degree of the Gauss map $\mathscr{G}_{H}:H\ra S^{n-1}$ is equal to $\frac{\chi(H)}{2}$ when $n$ is odd and to $\chi(W)$ when $n$ is even.
\end{theorem}
In the classical textbook \cite{GP},  Guillemin and Pollack show that the case of odd $n$   is a consequence of Poincar\'e-Hopf Theorem. Using  Morse Theory, Sakkalis \cite{S} gave a particularly  simple proof covering both cases of Hopf's Theorem \ref{HT}. In this note, we give another short proof, using the Chern-Gauss-Bonnet Theorem and its two versions: for manifolds with and without boundary \cite{Ch1,Ch2}.   While Guillemin and Pollack use  Hopf's  Theorem   to prove Chern-Gauss-Bonnet  for even-dimensional hypersurfaces,  we explain why we believe it might be as  natural to go the other way around and consider Theorem \ref{HT} to be an application of this other famous result.

One can look at Chern-Gauss-Bonnet as a combination of two fundamental facts in  topology:
\begin{itemize}
\item[(a)] a Poincar\'e Duality statement which says that the Pfaffian of the Levi-Civita connection (properly normalized) is Poincar\'e dual to the degree $0$ current induced by a generic vector field with only isolated stationary points; in simple terms the later is represented by a bunch of points with appropriate signs attached;  
\item[(b)] the homotopy invariance of the intersection number which takes the form of the Poincar\'e-Hopf Theorem, masterly explained in \cite{GP}; admitedly one needs also a particular construction  of a nice vector field in order to identify the index with the Euler characteristic.
\end{itemize}
Without entering into  historical details we will say that the first intrinsic proof of the generalization to the classical Gauss-Bonnet Theorem is due to Chern in \cite{Ch1}. This is where the transgression class (see below) first appears. A year later, Chern generalized his result to the case of a  manifold with boundary in \cite{Ch2}.
   
In regard to item (a), we recommend a fairly elementary treatment of a  general version of Chern-Gauss-Bonnet for oriented vector bundles of even rank in \cite{Ni}. For a direct proof, based on a general transgression formula see \cite{Ci}. We remark that, notwithstanding the difficulty of any complete proof,  the Chern-Gauss-Bonnet Theorem is not fundamentally based on Theorem \ref{HT} but is indeed a manifestation of Poincar\'e Duality. On the other hand, the latter is a direct consequence of the former as we will see next.

\section{The proof}
We use the notation of the previous section. The following is well-known to geometers. 
\begin{lemma} The pull-back via the Gauss map $\mathscr{G}_H:H\ra S^{n-1}$ of the Levi-Civita connection $\nabla^{S^{n-1}}$ on $S^{n-1}$ endowed with the round metric  is the Levi-Civita connection $\nabla^{H}$ on $H$.
\end{lemma}
\begin{proof} First of all, the pull-back (via the Gauss map) of the tangent bundle to the sphere is canonically  isomorphic to the tangent bundle to $H$. Let now $X:S^{n-1}\ra TS^{n-1}$ be a vector field on the sphere. Then at a point $p\in H$
\[\mathscr{G}_H^*\nabla^{S^{n-1}}_{Y_p}(X\circ \mathscr{G}_H):=\nabla^{S^{n-1}}_{d_p\mathscr{G}_H(Y_p)}X=P_{T_{\mathscr{G}_H(p)}S^{n-1}}dX(d_p\mathscr{G}_H(Y_p))=\]\[=P_{T_pH}d_p(X\circ \mathscr{G}_H)(Y_p)=\nabla_{Y_p}^H(X\circ \mathscr{G}_H).
\]
The notation $P_V$ represents orthogonal projection onto the linear subspace $V$. The pull-back connection is uniquely determined by its action on "pull-back" sections, i.e. sections of type $\widetilde{X}= X\circ \mathscr{G}_H$ and therefore we conclude that $\mathscr{G}_H^*\nabla^{S^{n-1}}=\nabla^H.$
\end{proof}
\begin{remark} This result  is a particular case of the celebrated Narasimhan Ramanan Theorem \cite{NR} which says that every pair vector bundle plus connection over a compact manifold is isomorphic with the pull-back of a universal bundle plus  universal connection via a classifying map. Indeed this holds even more generally for principal bundles with connections. In the Lemma above the sphere plays the role of the base space of the universal bundle and the Gauss map of a hypersurface plays the role of the classifying map. 
\end{remark}

In order to compute the degree of the Gauss map when $n$ is odd, instead of using the volume form on $S^n$ we use the Pfaffian of the Levi-Civita connection $\nabla^{S^{n-1}}$. By the Gauss-Bonnet Theorem the integral over the sphere of the Pfaffian $\Pf(\nabla^{S^{n-1}})$ (properly normalized) equals $2$. Hence the degree of the Gauss map is
\begin{equation}\label{parta} \deg{\mathscr{G}_{H}}=\frac{1}{2}\int_{H}\mathscr{G}_{H}^*\Pf(\nabla^{S^{n-1}})=\frac{1}{2}\int_{H}\Pf(\nabla^{H})=\frac{\chi(H)}{2}.
\end{equation}
where  we used Chern-Gauss-Bonnet Theorem on $H$ in the last equality. 
\vspace{0.5cm}

When $n$ is even and hence $S^{n-1}$ is odd-dimensional, the Pfaffian of $\nabla^{S^{n-1}}$ is zero by definition. However, there exists another natural form built with the help of the Levi-Civita connection that plays the role of the Pfaffian. This is a so-called transgression form. It can be described as follows. The trivial bundle $\bR^{n}\bigr|_{S^{n-1}}$ splits orthogonally as a direct sum of the normal bundle $\tau\ra S^{n-1}$ and $TS^{n-1}\ra S^{n-1}$. It can therefore be endowed with two connections: the trivial one $d$ and the direct sum connection $d\oplus \nabla^{S^{n-1}}$. 

The Pfaffians of both $\overline{\nabla}:=d$ and $\hat{\nabla}:=d\oplus \nabla^{S^{n-1}}$ are obviously identically zero as they are of degree $n>\dim{S^{n-1}}$.   However, a standard result of the theory of characteristic classes (see for example the article by Chern and Simons, \cite{CS}) says that the difference of the two Pfaffians is equal to an exact form,  which can be explicitly described. We write:
\[ 0=\Pf(d)-\Pf(d\oplus \nabla^{S^{n-1}})=d\TPf(\overline\nabla,\hat{\nabla})
\]
where 
\begin{equation}\label{transclass} \TPf(\overline\nabla,\hat{\nabla})=\int_{[0,1]}\Pf(\widetilde{\nabla}).
\end{equation}
Above $\widetilde{\nabla}:=\frac{d}{dt}+(1-t)\hat{\nabla}+t\overline{\nabla}$ is a  connection on the trivial bundle $\underline{\bR^{n}}\ra [0,1]\times S^{n-1}$ resulting by taking the affine combination of $\overline{\nabla}$ and $\hat{\nabla}$. The operator $\frac{d}{dt}$ signals that the covariant derivative of $\widetilde{\nabla}$ in the $\frac{\partial}{\partial t}$ direction is just derivation with respect to this field.  Finally, the integral is really integration of a form over the fiber of  the projection 
\[ \pi:[0,1]\times S^{n-1}\ra S^{n-1}.\]
 Before one starts believing that we just wrote $0$ in a complicated manner we show the following.
\begin{lemma}For $n$ even the folowing holds: \[\int_{S^{n-1}}\TPf(\overline\nabla,\hat{\nabla})=-1.\]
\end{lemma}
\begin{proof} Apply Chern-Gauss-Bonnet Theorem \cite{Ch2} to the manifold with boundary $D^{n}\subset \bR^{n}$:
\[ 1=\chi(D^{n})=\int_{D^{n}}\Pf(\overline{\nabla})-\int_{S^{n-1}}\TPf(\overline\nabla,\hat{\nabla})=-\int_{S^{n-1}}\TPf(\overline\nabla,\hat{\nabla}).
\] 
\end{proof}
\begin{remark}
For a computational proof and also for a local description of the transgression class the reader can take a look at Lemma 8.3.18 in \cite{Ni}. This starts by noting that  that $\TPf(\overline\nabla,\hat{\nabla})$ is a constant multiple of the volume form of $S^{n-1}$ since it enjoys the same symmetry property as the volume form. To see which multiple exactly, one can relate the two quantities at a single preferred point on the sphere. 
\end{remark}

It is easily seen that the transgression class satisfies a naturality property similar with that of the Pfaffian. Namely:
\[ \mathscr{G}_H^*\TPf(\overline\nabla,\hat{\nabla})=\TPf(\mathscr{G}_H^*\overline\nabla,\mathscr{G}_H^*\hat{\nabla})=\TPf(\overline{\nabla}^H,\hat{\nabla}^H),
\]
where $\overline\nabla^H$ is the trivial connection on $\underline\bR^n\ra H$ and $\hat{\nabla}^H$ is the direct sum of the trivial connection and the Levi-Civita connection on $H$. We conclude that
\begin{equation}\label{partb} \deg{ \mathscr{G}_H}=-\int_{H}\mathscr{G}_H^*\TPf(\overline\nabla,\hat{\nabla})=\int_{W}\Pf(\overline\nabla^{H})-\int_{\partial W}\TPf(\overline{\nabla}^H,\hat{\nabla}^H)=\chi(W),
\end{equation}
where, in the last line we used Chern-Gauss-Bonnet \cite{Ch2} on the oriented, compact, \emph{flat} manifold with boundary $W$. The lines (\ref{parta}) and (\ref{partb}) finish the proof of this note.

\end{document}